\documentclass[12pt]{article}

\usepackage{amsmath}
\usepackage{amsthm}
\usepackage{amsfonts}
\usepackage{mathrsfs}
\usepackage{stmaryrd}
\usepackage{setspace}
\usepackage{fullpage}
\usepackage{amssymb}
\usepackage{breqn}
\usepackage{enumitem}
\usepackage{bbold}
\usepackage{authblk}
\usepackage{comment}
\usepackage{hyperref}
\usepackage{pgf,tikz}
\usepackage{graphicx}
\usepackage{subcaption}

\newtheorem{thm}{Theorem}[section]
\newenvironment{thmbis}[1]
 {\renewcommand{\thethm} {\ref{#1}} \addtocounter{thm}{-1} \begin{thm}}
 {\end{thm}}

\newtheorem{conjecture}[thm]{Conjecture}

\newtheorem{prop}[thm]{Proposition}
\newenvironment{propbis}[1]
 {\renewcommand{\thethm} {\ref{#1}} \addtocounter{thm}{-1} \begin{prop}}
 {\end{prop}}

\newtheorem{clm}[thm]{Claim}

\newcommand\ex{\ensuremath{\mathrm{ex}}}

\newcommand\cA{{\mathcal A}}

\newcommand\cF{{\mathcal F}}

\newtheorem*{thm*}{Theorem}
\newtheorem*{prop*}{Proposition}
\newcommand{\ignore}[1]{}

\title{On hypergraph Tur\'an problems with bounded matching number}
\author{D\'{a}niel Gerbner\thanks{\small Alfr\'ed R\'enyi Institute of Mathematics, email:
\small \texttt{gerbner.daniel@renyi.hu}.},\hspace{1em}  Casey Tompkins\thanks{\small Alfr\'ed R\'enyi Institute of Mathematics, email:
\small \texttt{tompkins.casey@renyi.hu}.},\hspace{1em}
Junpeng Zhou\thanks{\small Department of Mathematics, Shanghai University, Shanghai 200444, P.R. China, email:
\small \texttt{junpengzhou@shu.edu.cn}.}}

\date{}

\begin{document}

\maketitle

\begin{abstract}
Very recently, Alon and Frankl, and Gerbner studied the maximum number of edges in $n$-vertex $F$-free graphs with bounded matching number, respectively. We consider the analogous Tur\'{a}n problems on hypergraphs with bounded matching number, and we obtain some exact results.
\end{abstract}

{\noindent{\bf Keywords}: Tur\'{a}n number, hypergraph, matching number}

{\noindent{\bf AMS subject classifications:} 05C35, 05C65}

\section{Introduction}
A \textit{hypergraph} $H=(V(H),E(H))$ consists of a vertex set $V(H)$ and a hyperedge set $E(H)$, where each hyperedge in $E(H)$ is a nonempty subset of $V(H)$.
If $|e|=r$ for any $e\in E(H)$, then $H$ is called an \textit{$r$-uniform hypergraph} ($r$-graph for short).
The \textit{degree} $d_H(v)$ (or simply $d(v)$ when there is no risk of confusion) of a vertex $v$  is the number of hyperedges containing $v$ in $H$.

Given a family $\mathcal{F}$ of $r$-graphs, a hypergraph $H$ is called \textit{$\mathcal{F}$-free} if $H$ does not contain any member of $\mathcal{F}$ as a subhypergraph.
The \textit{Tur\'{a}n number} ${\rm{ex}}_r(n,\mathcal{F})$ of $\mathcal{F}$ is the maximum number of hyperedges in an $\mathcal{F}$-free $r$-graph on $n$ vertices. If $\mathcal{F}=\{F\}$, then we write ${\rm{ex}}_r(n,F)$ instead of ${\rm{ex}}_r(n,\{F\})$. When $r=2$ we sometimes write ${\rm{ex}}(n,\mathcal{F})$ in place of ${\rm{ex}}_2(n,\mathcal{F})$.

Tur\'{a}n problems on graphs and hypergraphs are central topics in extremal combinatorics~\cite{Fu2}.
For some surveys on Tur\'{a}n problems on graphs and hypergraphs, one can refer to \cite{Fu1,Fu2,Ke1}.

Let $G(n,\ell,s)$ denote the complete $\ell$-partite graph on $n$ vertices with one part of order $n-s$ and each other part of order $\lfloor\frac{s}{\ell-1}\rfloor$ or $\lceil\frac{s}{\ell-1}\rceil$.
Recently, Alon and Frankl~\cite{af} considered Tur\'{a}n problems on graphs with bounded matching number.
For $s\ge s_0(F)$ and $n\ge n_0(F)$, they showed that
\begin{equation} \label{AlonFrankleq}
{\rm{ex}}(n,\{M_{s+1},F\})=|E(G(n,\ell,s))|,
\end{equation}
where $F$ is an arbitrary color-critical graph of chromatic number $\ell+1$. 
In~\cite{gerbner}, Gerbner generalized~\eqref{AlonFrankleq} as follows.

\begin{thm}\label{thm1}(\cite{gerbner})
If $\chi(F)>2$ and $n$ is sufficiently large, then $\ex(n,\{F,M_{s+1}\})=\ex(s,\cF)+s(n-s)$, where $\cF$ is the family of graphs obtained by deleting an independent set from $F$.
\end{thm}

In the case $F$ is bipartite, Gerbner~\cite{gerbner} also determined ${\rm{ex}}(n,\{M_{s+1},F\})$ apart from an additive constant term. 
Given a bipartite graph $F$, let $p(F)$ denote the smallest possible order of a color class in a proper 2-coloring of $F$. Gerbner~\cite{gerbner} proved the following result.

\begin{prop}\label{prop2}(\cite{gerbner})
If $F$ is bipartite and $p=p(F)\le s$, then $\ex(n,\{F,M_{s+1}\})=(p-1)n+O(1)$. Moreover, there is an integer $K=K(F,s)$ such that for any $n$, there is an $n$-vertex $\{F,M_{s+1}\}$-free graph $G$ with $|E(G)|=\ex(n,\{F,M_{s+1}\})$ that has vertices $v_1,\dots,v_{p-1}$ and at least $n-K$ vertices $u$ such that the neighborhood of $u$ is $\{v_1,\dots,v_{p-1}\}$. 
Furthermore, the vertices with neighborhood different from $\{v_1,\dots,v_{p-1}\}$ each have degree at least $p$.
\end{prop}

Hypergraph Tur\'{a}n problems are notoriously more difficult than graph versions.
In 1965, Erd\H{o}s \cite{Er1} proposed the Erd\H{o}s-Matching Conjecture on the Tur\'an number of $r$-uniform matchings.
It is known to hold for sufficiently large $n$, and we state a recent result due to Frankl~\cite{Fr1}.

\begin{thm}[\cite{Fr1}]\label{emc}
Let $r,s\geq1$ and $n\geq (2s+1)r-s$. Then $\ex_r(n,M_{s+1}^r)=\sum_{i=1}^s\binom{s}{i}\binom{n-s}{r-i}$. Equality holds only for families isomorphic to $\mathcal{A}(n,r,s)$, where $\mathcal{A}(n,r,s)$ is the $r$-graph such that each hyperedge intersects a fixed $s$-set. 
\end{thm}

In this paper we will study Tur\'{a}n problems on hypergraphs with bounded matching number. 
Let us start with some definitions related to hypergraphs.
A \textit{weakly independent set} of a hypergraph $H$ is a subset of $V(H)$  that does not contain a hyperedge. 
A \textit{proper $k$-coloring} of $H$ is a mapping from $V(H)$ to a set of $k$ colors such that no hyperedge is monochromatic. The \textit{chromatic number} $\chi(H)$ of $H$ is the minimum colors needed for proper coloring $H$.
From now on, we use $F$ to denote a given $r$-graph.
First we prove the following hypergraph version of Theorem~\ref{thm1}.

\begin{thm}\label{thmnew1}
If $\chi(F)>2$ and $n$ is sufficiently large, then $\ex_r(n,\{F,M_{s+1}^r\})=\ex_r(s,\cF)+\sum_{i=1}^t\binom{s}{i}\binom{n-s}{r-i}$, where $t=\min\{s,r-1\}$ and $\cF$ is the family of $r$-graphs obtained by deleting a weakly independent set from $F$.
\end{thm}

Further, we determine the order of magnitude of $\ex_r(n,\{F,M_{s+1}^r\})$ for a wide class of hypergraphs, and we express how it depends on the extremal number of some link hypergraphs of $F$ in the remaining cases. 
The \textit{link hypergraph} of a vertex $v$ of $F$ is the $(r-1)$-graph that has an $(r-1)$-set $e$ as a hyperedge if and only if $e$ together with $v$ forms a hyperedge of $F$.

\begin{prop}\label{thmnew2}
If $F$ does not contain a vertex which is contained in every hyperedge, then $\ex_r(n,\{F,M_{s+1}^r\})=\Theta(n^{r-1})$. 
Otherwise, let $F'$ denote the link $(r-1)$-graph of a vertex that is contained in each hyperedge. Then we have $\ex_r(n,\{F,M_{s+1}^r\})=\Theta(\ex_{r-1}(n,F'))$.
\end{prop}

Next let us consider the case $\chi(F)=2$. We will take the colors of a proper $2$-coloring of $F$ to be red and blue and when we say \textit{red-blue coloring} we always mean a proper coloring.
Let $p(F)$ denote the minimum number of red vertices in a red-blue coloring of $F$. 
We say that a red-blue coloring of $F$ is \textit{strong} if each hyperedge contains exactly one red vertex. 
We remark that the set of red vertices in a strong red-blue coloring is also called a \textit{crosscut} in the literature.
Let $q(F)$ denote the minimum number of red vertices in a strong red-blue coloring of $F$, with $q(F)=\infty$ if $F$ does not have a strong red-blue coloring. 
It is easy to see that $p(F)\leq q(F)$. 
Let $M(F)$ denote the matching number of $F$. For the $2$-chromatic case, we have the following result. 

\begin{prop}\label{thmnew3}
Let $\chi(F)=2$, $n\geq (2s+1)r-s$ and $\omega=\min\{s,M(F)-1\}$. Then 
\begin{eqnarray*}
\sum_{i=1}^{\omega}\binom{\omega}{i}\binom{n-\omega}{r-i}\leq \ex_r(n,\{F,M_{s+1}^r\})\leq \sum_{i=1}^s\binom{s}{i}\binom{n-s}{r-i}.
\end{eqnarray*}
Moreover, if $s<\min\{p(F),r\}$, then $\ex_r(n,\{F,M_{s+1}^r\})=\sum_{i=1}^s\binom{s}{i}\binom{n-s}{r-i}$. If $q(F)>s$, then $\ex_r(n,\{F,M_{s+1}^r\})=s\binom{n}{r-1}+O(n^{r-2})$.
\end{prop}

For $q(F)\leq s$ and $r=3$, we give a conjecture (see Conjecture \ref{con}), and prove that this conjecture holds for a special class of 2-chromatic hypergraphs.
For a given graph $G$, the \textit{$r$-expansion} $G^r$ of $G$ is the $r$-graph obtained from $G$ by adding $r-2$ distinct new vertices to each edge of $G$, such that the $(r-2)|E(G)|$ new vertices are distinct from each other and are not in $V(G)$. Expansions were introduced by Mubayi \cite{mub}, see \cite{mubver} for a survey. 
Very recently, Zhou and Yuan \cite{zhy} considered linear Tur\'{a}n problems on expansions of graphs with bounded matching number. 
If $r\geq3$, then by coloring $V(G)$ red and $V(G^r)\setminus V(G)$ blue, it can be seen that $G^r$ has a proper 2-coloring. Therefore, $\chi(G^r)=2$ for $r\geq3$, i.e., the $r$-expansion of any graph is a 2-chromatic hypergraph. 
We obtain the following results on expansions of bipartite graphs.

\begin{thm}\label{thmnew5}
If $G$ is bipartite and $p=p(G)\leq s$, then for sufficiently large $n$,
\begin{eqnarray*}
\ex_r(n,\{G^r,M_{s+1}^r\})=(p-1)\binom{n}{r-1}+O(n^{r-2}).
\end{eqnarray*}
\end{thm}

\begin{prop}\label{thmnew6}
If $G$ is bipartite and $p=p(G)>s$, then for sufficiently large $n$,
\begin{eqnarray*}
\ex_r(n,\{G^r,M_{s+1}^r\})=\sum_{i=1}^s\binom{s}{i}\binom{n-s}{r-i}.
\end{eqnarray*}
\end{prop}

For expansions of non-bipartite graphs, the situation is more complicated. Let $\chi(G)=k>2$. We determine the asymptotics when $k<r$. For a graph $G$, let $W$ be an independent set of vertices in $G$ and $Z$ be the set of edges not incident to $W$. Let $m(G)$ denote the smallest $|W|+|Z|$ for sets satisfying the above properties. 

\begin{thm}\label{temi}
    Let $k< r$. Then $\ex_r(n,\{G^r,M_{s+1}^r\})=(1+o(1))\min\{m(G)-1,s\}\binom{n}{r-1}$.
\end{thm}

\section{Proof of Theorem \ref{thmnew1}}
Now we are ready to prove Theorem \ref{thmnew1}, which we restate here for convenience.

\begin{thmbis}{thmnew1}
If $\chi(F)>2$ and $n$ is large enough, then $\ex_r(n,\{F,M_{s+1}^r\})=\ex_r(s,\cF)+\sum_{i=1}^t\binom{s}{i}\binom{n-s}{r-i}$, where $t=\min\{s,r-1\}$ and $\cF$ is the family of $r$-graphs obtained by deleting a weakly independent set from $F$.
\end{thmbis}
\begin{proof}[{\bf{Proof}}]
Note that there are no empty graphs in $\cF$ as $\chi(F)>2$. Thus we may assume that $H_0$ is an $s$-vertex $\cF$-free $r$-graph with $\ex_r(s,\cF)$ hyperedges. Let us add $n-s$ new vertices and add each hyperedge that contains at least one and at most $r-1$ vertices of $H_0$. The resulting $r$-graph is $M_{s+1}^r$-free as $s$ vertices of $H_0$ are incident to all the hyperedges, and $F$-free as $H_0$ is $\cF$-free and the $n-s$ vertices outside $H_0$ form a weakly independent set. This gives the lower bound.




To show the upper bound, consider an $\{F,M_{s+1}^r\}$-free $n$-vertex $r$-graph $H$. Let $v_1,\dots,v_n$ be the vertices of $H$ listed in decreasing order of their degrees. Observe that $d(v_{s+1})\le rs\binom{n-2}{r-2}$. Indeed, otherwise we may greedily find a matching $M_{s+1}^r$ in the following way. In step $i$ ($1\leq i\leq s+1$), we pick a hyperedge containing $v_i$ and no other vertex from $\{v_1,v_2,\dots,v_{s+1}\}$ which is disjoint from all previously selected hyperedges.
This is possible since the union of the previously selected hyperedges and $\{v_1,v_2,\dots,v_{s+1}\}\setminus \{v_i\}$ has size at most $rs$, and for any vertex $v \neq v_i$, at most $\binom{n-2}{r-2}$ hyperedges contain both $v$ and $v_i$.

Observe also that $H$ has at most $\sum_{i=1}^{rs}d(v_i)\le \sum_{i=1}^{s}d(v_i)+r(r-1)s^2\binom{n-2}{r-2}$ hyperedges. Indeed, the at most $rs$ vertices of a largest matching are incident to every hyperedge, and $rs$ vertices are incident to at most $\sum_{i=1}^{rs}d(v_i)$ hyperedges. The upper bound in this inequality follows from $d(v_{s+1})\leq rs\binom{n-2}{r-2}$.

We claim that $d(v_s)\ge \binom{n-1}{r-1}-cn^{r-2}$ for some $c>0$. Indeed, otherwise for any $c$ that satisfies $r(r-1)s^2\binom{n-2}{r-2}-cn^{r-2}\leq 0$, we have
\begin{eqnarray*}
\sum_{i=1}^{s}d(v_i)+r(r-1)s^2\binom{n-2}{r-2}&\le& s\binom{n-1}{r-1}-cn^{r-2}+r(r-1)s^2\binom{n-2}{r-2} \\
&\le& \sum_{i=1}^t\binom{s}{i}\binom{n-s}{r-i},
\end{eqnarray*}
and we are done.  Let $U=\{v_1,\dots,v_s\}$, 
and for $1\leq i\le s$ let $H_i$ denote the $(r-1)$-graph on $V(H)\setminus U$ which contains an $(r-1)$-set $h$ as a hyperedge if and only if $h\cup \{v_i\}$ is a hyperedge of $H$. By the above, there exists some $c'\geq c$ such that there are at least $\binom{n-1}{r-1}-c'n^{r-2}$ hyperedges in $H_i$ for any $1\leq i\leq s$. Indeed, we subtract from $d(v_i)$ the number of hyperedges that contain $v_i$ and another vertex of $U$. Then we have that all but at most $sc' n^{r-2}$ of the $(r-1)$-sets in $V(H)\setminus U$ are in each $H_i$. Since $n$ is large enough, there is a $K^{(r)}_{r(s+1)}$ contained in each $H_i$, where $K^{(r)}_{r(s+1)}$ is the complete $r$-graph on $r(s+1)$ vertices (using the known bounds on the Tur\'an number of $K^{(r)}_{r(s+1)}$). 

We claim that there is no hyperedge of $H$ outside $U$. Indeed, otherwise we could find $M^r_{s+1}$ greedily as earlier: first we pick a hyperedge outside $U$, and then in step $i+1$ ($1\leq i\leq s$), we pick $v_i$ and an $(r-1)$-set from the $K^{(r)}_{r(s+1)}$ found earlier that avoids all the vertices we have already picked. This is doable since there are at most $rs$ vertices already picked.

If $H[U]$ is $\cF$-free, then the number of hyperedges of $H$ is at most $\ex_r(s,\cF)+\sum_{i=1}^t\binom{s}{i}\binom{n-s}{r-i}$, where the first term is an upper bound on the number of hyperedges inside $U$, while the second term is an upper bound on the number of hyperedges with at least one vertex inside $U$ and at least one vertex outside $U$. Assume now that $H[U]$ contains an $r$-graph $F_0\in\cF$ with vertex-set $Q$. Then for each $(|V(F)|-|V(F_0)|)$-set $Q'$ outside $U$, since it does not extend the copy of $F_0$ to a copy of $F$, there is an $r$-subset of $Q\cup Q'$ that is not a hyperedge of $H$. Therefore, the number of hyperedges of $H$ is at most $\binom{s}{r}+\sum_{i=1}^t\binom{s}{i}\binom{n-s}{r-i}-(n-s)/(|V(F)|-|V(F_0)|)$, where the the first term is an upper bound on the number of hyperedges inside $U$, while the other terms are upper bound on the number of hyperedges with at least one vertex inside $U$ and at least one vertex outside $U$. Since $n$ is large enough, we have $\binom{s}{r}-(n-s)/(|V(F)|-|V(F_0)|)\leq \ex_r(s,\cF)$, and the proof is complete. 
\end{proof}

In the following we prove Proposition~\ref{thmnew2}, which we restate here for convenience.

\begin{propbis}{thmnew2}
Suppose that $F$ does not contain a vertex that is in each hyperedge, then $\ex_r(n,\{F,M_{s+1}^r\})=\Theta(n^{r-1})$.
Otherwise, let $F'$ denote the link $(r-1)$-graph of a vertex that is contained in each hyperedge. 
Then we have $\ex_r(n,\{F,M_{s+1}^r\})=\Theta(\ex_{r-1}(n,F'))$.
\end{propbis}
\begin{proof}[{\bf{Proof}}]
The upper bound $O(n^{r-1})$ follows from Theorem \ref{emc}.
The matching lower bound is given by the $n$-vertex $r$-graph with hyperedge set consisting of all hyperedges containing a given vertex $v$.

In the case $F$ contains a vertex $x$ in each hyperedge, the lower bound is given by taking an $F'$-free $(r-1)$-graph on $n-1$ vertices with $\ex_{r-1}(n,F')$ hyperedges, and adding a new vertex $v$ to each hyperedge. Obviously, this $r$-graph is $M_{s+1}^r$-free and $F$-free.
Indeed, otherwise the link hypergraph of $v$ contains a
copy of $F'$ as $v$ appears in each hyperedge of a copy of $F$.

We will show that $\ex_r(n,\{F,M_{s+1}^r\})\le rs\cdot\ex_{r-1}(n-1,F')$. Consider an $\{F,M_{s+1}^r\}$-free $r$-graph $H$ and a largest matching in $H$. The largest matching has at most $rs$ vertices, and each other hyperedge contains at least one of these vertices. Clearly, the link hypergraph of each vertex of the largest matching is $F'$-free, thus each vertex is in at most $\ex_{r-1}(n-1,F')$ hyperedges, completing the proof.
\end{proof}

\section{2-chromatic case}
\subsection{General hypergraphs}
Recall that for an $r$-graph $F$, $p(F)$ denotes the smallest possible number of red vertices in a red-blue coloring of $F$ and $q(F)$ denotes the smallest possible number of red vertices in a strong red-blue coloring of $F$, with $q(F)=\infty$ if $F$ does not have a strong red-blue coloring.

Let us prove Proposition \ref{thmnew3} that we restate here for convenience.

\begin{propbis}{thmnew3}
Let $\chi(F)=2$, $n\geq (2s+1)r-s$ and $\omega=\min\{s,M(F)-1\}$. Then 
\begin{eqnarray*}
\sum_{i=1}^{\omega}\binom{\omega}{i}\binom{n-\omega}{r-i}\leq \ex_r(n,\{F,M_{s+1}^r\})\leq \sum_{i=1}^s\binom{s}{i}\binom{n-s}{r-i}.
\end{eqnarray*}
Moreover, if $s<\min\{p(F),r\}$, then $\ex_r(n,\{F,M_{s+1}^r\})=\sum_{i=1}^s\binom{s}{i}\binom{n-s}{r-i}$. If $q(F)>s$, then $\ex_r(n,\{F,M_{s+1}^r\})=s\binom{n}{r-1}+O(n^{r-2})$.
\end{propbis}

\begin{proof}[{\bf{Proof}}]
It is easy to see that forbidding a matching gives the upper bound $\sum_{i=1}^s\binom{s}{i}\binom{n-s}{r-i}$. 
For the lower bound, we consider the following hypergraph: taking a set $A$ of order $\omega$ and a set $B$ of order $n-\omega$, and each hyperedge that contains at least one vertex of $A$. Clearly, this resulting hypergraph is $M_{\omega+1}^r$-free as $\omega$ vertices of $A$ are incident to all the hyperedges. Since $M^r_{\omega+1}\subseteq F$ as $\omega+1\leq M(F)$, this resulting hypergraph is $F$-free. 

If $s<\min\{p(F),r\}$, then the construction from Theorem \ref{emc} is $F$-free and $M^r_{s+1}$-free, giving a sharp lower bound.



If $q(F)>s$, then forbidding a matching already gives the upper bound $s\binom{n}{r-1}+O(n^{r-2})$. For the lower bound, we consider the following hypergraph: taking a set $A$ of order $s$ and a set $B$ of order $n-s$, and each hyperedge that contains exactly one $A$. Clearly, this hypergraph is $M^r_{s+1}$-free and $F$-free as $q(F)>s$. Therefore, $\ex_r(n,\{F,M^r_{s+1}\})=s\binom{n}{r-1}+O(n^{r-2})$.
\end{proof}

Consider now the case $q=q(F)\le s$ and $r=3$. Let us fix a strong red-blue coloring such that the set $Q$ of red vertices is of order $q$. Let $F_v$ denote the link graph of $v$. Observe that if $v\not\in Q$, then $Q$ is an independent vertex cover in $F_v$, thus $\chi(F_v)=2$. Let us order the link graphs of vertices in $Q$ according to their chromatic number the following way: $\chi(F_1)\ge \chi(F_2)\ge \cdots \ge \chi(F_q)$. Let $H_i$ for $i\le q$ denote the following hypergraph. Let $A_1$ be a set of $i-1$ vertices, $A_2$ be a set of $s-i+1$ vertices and $B$ be a set of $n-s$ vertices. We take a copy of $T(n-s,\ell_i-1)$ on $B$, where $\ell_i=\chi(F_i)$. Then we take each hyperedge that contains a vertex of $A_1$ and two vertices of $B$, and each hyperedge that contains a vertex of $A_2$ and an edge of the $T(n-s,\ell_i-1)$ in $B$.
Obviously, $H_i$ is $M_{s+1}^3$-free as each edge of $H_i$ contains at least one vertex in $A_1\cup A_2$. Observe that if $\ell_i>2$, then the link graph of a vertex in $A_2\cup B$ is at most $(\ell_i-1)$-chromatic. Indeed, the link graph of a vertex in $A_2$ is $T(n-s,\ell_i-1)$, while the link graph of a vertex in $B$ has one class replaced by the independent set $A_1\cup A_2$, i.e., it is 2-chromatic. This shows that $H_i$ is $F$-free as $F$ contains at least $i$ vertices with link graph of chromatic number at least $\ell_i$. 
If $\ell_i=2$, then there are no hyperedges in $A_2\cup B$ as $T(n-s,1)$ is an empty graph. By coloring $A_1$ red and the other vertices blue, $H_i$ has a strong red-blue coloring with less than $q$ red vertices, thus so does every subgraph of $H_i$. 

By the definition of $q(F)$, $F$ does not have such a coloring, thus $H_i$ is $F$-free. This gives a lower bound on $\ex_3(n,\{F,M_{s+1}^3\})$. We propose the following conjecture.

\begin{conjecture}\label{con}
If $\chi(F)=2$, $q(F)\le s$ and $n$ is sufficiently large, then $\ex_3(n,\{F,M_{s+1}^3\})=\max\{|E(H_i)|:\, i\le q(F)\}+o(n^2)$.
\end{conjecture}

If $G$ is a bipartite graph with $p=p(G)\leq s$, then $q(G^3)=p\leq s$ and $\chi(G^3_1)=\chi(G^3_2)=\cdots=\chi(G^3_p)=2$. By the definition of $H_i$, we have $\max\{|E(H_i)|:\, i\le q(F)\}=(p-1)\binom{n}{2}$. Thus, Theorem \ref{thmnew5} for $r=3$ confirms Conjecture \ref{con} for the case of any bipartite graph $G$ with $p(G)\leq s$ and sufficiently large $n$.




\subsection{Expansion of graphs with bounded matching number} 

Let us prove Theorem~\ref{thmnew5}, which we restate here for convenience.

\begin{thmbis}{thmnew5}
If $G$ is bipartite graph and $p=p(G)\leq s$, then for sufficiently large $n$,
\begin{eqnarray*}
\ex_r(n,\{G^r,M_{s+1}^r\})=(p-1)\binom{n}{r-1}+O(n^{r-2}).
\end{eqnarray*}
\end{thmbis}
\begin{proof}[{\bf{Proof.}}]
For the lower bound, we consider the following hypergraph: take a set $A$ of order $p-1$ and a set $B$ of order $n-p+1$, and each hyperedge that contains exactly one vertex of $A$.
Clearly, this resulting hypergraph is $M_{s+1}^r$-free as $p-1$ vertices of $A$ are incident to all the hyperedges, and $G^r$-free as $q(G^r)=p$. Therefore, $\ex_r(n,\{G^r,M^r_{s+1}\})\geq(p-1)\binom{n}{r-1}+O(n^{r-2})$.

We now continue with the upper bound. Suppose $H$ is a $\{G^r,M^r_{s+1}\}$-free $r$-graph on $n$ vertices. Let $M$ be a largest matching in $H$ and $U:=V(M)$. Then $|U|\leq rs$ as $H$ is $M_{s+1}^r$-free. 
Let $q:=|V(G)|-p$.
Consider a $p$-subset $S$ of $U$. We say that a vertex $v\not\in U$ is \textit{good} with respect to $S$ if for every $u\in S$ there are at least $pqrn^{r-3}$ hyperedges in $H$ that contain both $u$ and $v$. We claim that there are at most $q-1$ good vertices with respect to $S$. Assume otherwise, identify the vertices of $G$ in a part of order $p$ with $S$ and identify the other vertices of $G$ with $q$ good vertices and pick an arbitrary order of the edges of $G$. For each edge we will pick a hyperedge of $H$ that contains the edge such that altogether they form $G^r$. At each step, we have to pick a hyperedge that contains $u$, $v$ and does not contain any of the at most $pqr$ other vertices in the hyperedges picked earlier. This is possible, since we forbid less than $pqr\binom{n}{r-3}$ hyperedges this way.

As there are at most $\binom{rs}{p}$ $p$-sets in $U$, there are at most $(q-1)\binom{rs}{p}$ vertices that are good with respect to some $p$-sets of $U$. Let $W$ be the set of at least $n-(q-1)\binom{rs}{p}-rs$ other vertices outside $U$. Given a vertex $w\in W$, we say that $u\in U$ is \textit{nice} with respect to $w$ if at least  $pqrn^{r-3}$ hyperedges in $H$ contain both $u$ and $w$. There are at most $p-1$ nice vertices with respect to $w$, otherwise $w$ would be good with respect to a $p$-set.

There are $O(n^{r-2})$ hyperedges that contain fewer than $r-1$ vertices from $W$. We claim that $V(H)\setminus U$ is a weakly independent set of $H$. In fact, if there is a hyperedge that contains $r$ vertices from $V(H)\setminus U$, then we may find a larger matching in $H$, which contradicts that $M$ is a largest matching. Thus, each other hyperedge contains a vertex $w\in W$, a vertex $u\in U$ and $r-2$ other vertices from $W$. For vertices of $U$ that are not nice with respect to $w$, there are less than  $pqrn^{r-3}$ hyperedges in $H$ that contain both $u$ and $w$. Altogether, there are at less than $|U||W|pqrn^{r-3}=O(n^{r-2})$ hyperedges of this type.

Finally, each of the remaining hyperedges contains a vertex $w\in W$, one of the at most $p-1$ nice vertices with respect to $w$ and $r-2$ other vertices from $W\setminus\{w\}$. Let $H'$ be the subgraph of $H$ consisting of these remaining hyperedges. For any $w\in W$, we have $d_{H'}(w)\leq (p-1)\binom{n}{r-2}$. Then
\begin{eqnarray*}
|E(H')|=\frac{\sum_{w\in W}d_{H'}(w)}{r-1}\leq \frac{|W|\cdot(p-1)\binom{n}{r-2}}{r-1}\leq \frac{n(p-1)\binom{n}{r-2}}{r-1}=(p-1)\binom{n}{r-1}+O(n^{r-2}).
\end{eqnarray*}
This completes the proof.
\end{proof}

Recall that Proposition \ref{thmnew6} states that if $G$ is bipartite and $p=p(G)>s$, then for sufficiently large $n$,
$\ex_r(n,\{G^r,M_{s+1}^r\})=\sum_{i=1}^s\binom{s}{i}\binom{n-s}{r-i}.$

\begin{proof}[{\bf{Proof of Proposition \ref{thmnew6}.}}]
The upper bound follows from Theorem \ref{emc}.
For the lower bound, we consider the following hypergraph: taking a set $A$ of order $s$ and a set $B$ of order $n-s$, and each hyperedge that contains at least one vertex of $A$. Clearly, this resulting hypergraph is $M_{s+1}^r$-free as $s$ vertices of $A$ are incident to all the hyperedges. Note that $M_{p}\subseteq G$ by the definition of $p(G)$, thus $M^r_{p}\subseteq G^r$. So this resulting hypergraph is $G^r$-free as $s+1\leq p$.
Therefore, $\ex_r(n,\{G^r,M^r_{s+1}\})=\sum_{i=1}^s\binom{s}{i}\binom{n-s}{r-i}$.
\end{proof}

Before we give the proof of Theorem~\ref{temi}, we introduce some definitions.
We say that a pair of vertices $u,v$ is \textit{heavy} if there are at least 
$(r-2)|E(G)|+|V(G)|$ hyperedges with pairwise intersection equal to $\{u,v\}$.
In particular, if we already embedded the expansion of a subgraph of $G$ and we want to embed another hyperedge of $G^r$ that corresponds to the edge $uv$, if $u,v$ is a heavy pair, we can do this greedily. Indeed, we have to avoid picking a hyperedge that contains any of the at most $(r-2)(|E(G)|-1)+|V(G)|$ other already embedded vertices. Observe also if $u,v$ is not heavy, then by deleting $u,v$ from the hyperedges that contain both $u$ and $v$, we obtain an $(r-2)$-graph with no matching of size $(r-2)|E(G)|+|V(G)|$, thus there are $O(n^{r-3})$ hyperedges containing both $u$ and $v$ by Theorem \ref{emc}.
Let $K(i)$ denote the complete $i$-partite graph with parts of order $3|E(G)|$.

Recall that for a graph $G$, $W$ is an independent set of vertices in $G$, $Z$ is the set of edges not incident to $W$, and $m(G)$ denote the smallest $|W|+|Z|$ for sets satisfying the above properties. 
Let us continue with the proof of Theorem \ref{temi}, which states that $\ex_r(n,\{G^r,M_{s+1}^r\})=(1+o(1))\min\{m(G)-1,s\}\binom{n}{r-1}$ for $\chi(G)=k<r$. 

\begin{proof}[{\bf Proof of Theorem \ref{temi}}]   The lower bound is given by $\cA(n,r,s)$ or by the following $r$-graph $\mathcal{G}_1$. We take a set $A$ of $m(G)-1$ vertices and a set $C$ of $n-m(G)+1$ vertices. We take as hyperedges of $\mathcal{G}_1$ the $r$-sets that contain a vertex from $A$ and $r-1$ vertices from $C$.

We claim that $\mathcal{G}_1$ is $G^r$-free. Assume otherwise and let $G$ denote a core of a $G^r$ in $\mathcal{G}_1$. Clearly $G$ intersects $A$ in an independent set, let us denote it by $W$. Let $Z$ denote the set of edges of $G$ of not incident to $W$, then each hyperedge corresponding to an edge in $Z$ contains a vertex of $A$ such that they are distinct from each other and not in $W$. This implies that $|A|\ge m(G)$, a contradiction.

In the case $s\le m(G)-1$, the upper bound follows from Theorem \ref{emc}, thus we can assume that $s\ge m(G)$. Let $H$ be an $n$-vertex $\{G^r,M_{s+1}^r\}$-free $r$-graph and $U=\{v_1,\dots,v_m\}$ be a set of minimum size with the property that each hyperedge contains a vertex from $U$. Then $|U|\le rs$. 

For each non-heavy pair of vertices $u,v$ with $u\in U$, $v\not\in U$, we remove each hyperedge containing both $u$ and $v$. There are $O(n)$ such pairs, thus we remove $O(n^{r-2})$ hyperedges this way. Let $H'$ be the resulting hypergraph. Let $H_i$ be the $(r-1)$-graph on $V(H)\setminus U$ with an $(r-1)$-set $Q$ being a hyperedge if and only if $Q$ with $v_i$ forms a hyperedge of $H'$, i.e., the link hypergraph of $v_i$ after deleting the hyperedges containing at least two vertices from $U$. Note that if a hyperedge is in $H_i$, then each vertex of that hyperedge forms a heavy pair with $v_i$.
Let $U(Q)$ denote the set of vertices in $U$ that have $Q$ in their link.
     
\begin{clm}
    For any set $U'$ with $|U'|\ge m(G)$, there are $o(n^{r-1})$ $(r-1)$-sets $Q$ with $U'=U(Q)$. 
\end{clm}
     
\begin{proof}[\bf Proof of Claim] 
Let $\tilde{H}$ denote the $(r-1)$-graph consisting of all $(r-1)$-sets $Q$ with $U(Q)=U'$. If $\tilde{H}$ is $K(k)^{r-1}$-free, then by a result of Erd\H os \cite{Er2}, $\tilde{H}$ has $o(n^{r-1})$ hyperedges, as observed in \cite{mubver}. Therefore, we can assume that there is a $K(k)^{r-1}$ in $\tilde{H}$. Consider $W$ and $Z$ from the definition of $m(G)$. Let $G'$ denote the graph we obtain from $G$ by deleting $W$, i.e., $E(G')=Z$. First we embed the at most $k$ color classes of a proper coloring of $G'$ into $K(k)$, each part to a different part of $K(k)$, arbitrarily. Then for each edge of $G'$ there is a hyperedge in $K(k)^{r-1}$ that extends it. 
Then we consider these hyperedges and for each such hyperedge, we choose greedily a new vertex from $U'$. This can be done as at most $m(G)-1$ vertices were chosen before. Now we embed the vertices of $W$ into the remaining part of $U'$ arbitrarily. Recall that we partitioned $V(G)$ into an independent set $W$ and $G'$, thus there are at least $m(G)-|Z|\ge |W|$ unused vertices in $U'$. For each such vertex $u\in U'$ and each vertex $v$ of $G'$ embedded into $K(k)$, the pair $u,v$ must be heavy, thus we can greedily extend the embedding to an embedding of $G^r$, a contradiction. 
\end{proof}

Let us return to the proof of the theorem. The number of hyperedges which intersect $U$ in $0$ or at least $2$ vertices is $O(n^{r-2})$. The number of hyperedges we deleted from $H$ to obtain $H'$ is $O(n^{r-2})$. We count the rest of the hyperedges according to the intersection $Q$ with $V(H)\setminus U$. For $Q$ with $|U(Q)|\le m(G)-1$, clearly there are at most $\binom{n}{r-1}$ ways to pick $Q$ and at most $m(G)-1$ ways to extend $Q$ to a hyperedge of $H$ with a vertex from $U$. For the other sets $Q$, there are $O(1)$ ways to choose $U(Q)$, and for any $U(Q)$ with $|U(Q)|\ge m(G)$ there are $o(n^{r-1})$ ways to choose the sets $Q$ by Claim~3.2. This completes the proof.
    \end{proof}

\section{Concluding remarks}

We have obtained sharp or nearly sharp bounds on $\ex_r(n,\{F,M_{s+1}^r\})$ for several hypergraphs~$F$. Let us discuss a missing case when $F=G^r$ for a graph $G$ with $\chi(G)=k\ge r$.

In the proof of Theorem \ref{temi}, we gave a construction $\mathcal{G}_1$ that was asymptotically optimal if $\chi(G)=k<r$.

In the case $k=r$, the situation is more complicated. It is possible that the edges in $Z$ form a graph of chromatic number $k-1$ (i.e., $W$ is a whole color class in a proper $k$-coloring of $F$). In this case the proof of Theorem \ref{temi} works in the same way and yields the same bound. However, it is possible that the edges in $Z$ form a graph of chromatic number $k$. In this case one can give another construction.

We consider a set of vertices $W$ and two sets of edges $X$ and $Y$ such that $W$ is an independent set not incident to any edge from $X\cup Y$, $X$ and $Y$ are disjoint, deleting $W$ and $X$ from $G$ results in a graph with chromatic number $k-1$, and deleting $W,X,Y$ results in a graph with chromatic number $k-2$. Let $Z$ be the set of edges not incident to $W$ that are not in $X\cup Y$. 
Let $f(W,X,Y,Z)=|X|\binom{n}{r-1}+|Y|\binom{k-1}{r-1}\left(\frac{n}{k-1}\right)^{r-1}+(|W|+|Z|-1)\binom{k-2}{r-1}\left(\frac{n}{k-2}\right)^{r-1}$. Let $h(G)$ denote the smallest possible value of $f(W,X,Y,Z)$ for sets $W,X,Y,Z$ satisfying the above properties, and let $w,x,y,z$ denote the sizes of those sets $W,X,Y,Z$, respectively.

Let $\mathcal{G}_1'$ be the following $r$-graph. We take a set $A$ of $w+z-1$ vertices, a set $B$ of $x$ vertices, a set $C$ of $y$ vertices and a set $D$ of $n-w-x-y-z+1$ vertices. We take on $D$ a complete $(k-2)$-partite $(r-1)$-graph $\mathcal{G}_0$ and a complete $(k-1)$-partite $(r-1)$-graph~$\mathcal{G}_0'$. We take as hyperedges of $\mathcal{G}_1'$ the $r$-sets that contain a vertex from $B$ and $r-1$ vertices from $D$, the $r$-sets that contain a vertex from $A$ and a hyperedge from $\mathcal{G}_0$ and the $r$-sets that contain a vertex from $C$ and a hyperedge from $\mathcal{G}_0'$.

\begin{prop}
$\mathcal{G}_1'$ is $G^r$-free.
\end{prop}
\begin{proof}[\bf Proof]
Assume otherwise, and let $G$ denote a core of a $G^r$ in $\mathcal{G}_1'$. Clearly $G^r$ intersects $A\cup B\cup C$ in an independent set, let us denote it by $W'$. We can partition $W'$ into the set $W$ of vertices that belong to the core $G$ and the set $Z$ of other vertices. Then $G$ intersects $D$ in a graph of chromatic number at least $k-1$. Therefore, $G$ has to contain some edges such that the corresponding hyperedges of $G^r$ intersect $D$ in non-edges of $\mathcal{G}_0$. Among these edges, let $Y$ be the set of those edges such that the corresponding hyperedges of $G^r$ that intersect $D$ belong to $\mathcal{G}_0'$ and $X$ be the set of edges such that the corresponding hyperedges of $G^r$ that intersect $D$ belong to neither $\mathcal{G}_0$ nor $\mathcal{G}_0'$. Then the hyperedges of $G^r$ that correspond to $X$ each contain a vertex from $B$. The hyperedges of $G^r$ that correspond to $Y$ each contain a vertex from $C$. The hyperedges of $G^r$ that correspond to the other edges not incident to $W$ each contain a vertex from $Z$. Therefore, $|W'|\le w+z-1$, $|X|\le x$ and $|Y|\le y$.  Observe that $W,X,Y,Z$ satisfy the partition properties defined above but $f(W,X,Y,Z)<h(G)$, a contradiction.\
\end{proof}

Consider now the case $k>r$. Note that $\mathcal{G}_1$ is also $G^r$-free for $k>r$. 
We give another construction. Let $U$ be an independent set of vertices in $F$ and $Z$ be a set of edges such that deleting $U$ and $Z$ from $F$ results in a graph with chromatic number $k-2$. Let $m'(G)$ denote the smallest $|Z|$ for sets satisfying the above properties.

Let $\mathcal{G}_2$ be the following $r$-graph. We take a set $A$ of $m'(G)-1$ vertices, a set $B$ of $s-m'(G)+1$ vertices and a set $C$ of $n-s$ vertices. We take on $C$ a complete $(k-2)$-partite $(r-1)$-graph $\mathcal{G}_0$. We take as hyperedges of $\mathcal{G}_2$ the $r$-sets that contain a vertex from $A$ and $r-1$ vertices from $C$, and the $r$-sets that contain a vertex from $B$ and a hyperedge from $\mathcal{G}_0$.

\begin{prop}
$\mathcal{G}_2$ is $G^r$-free. 
\end{prop}
\begin{proof}[\bf Proof]
Assume otherwise and let $G$ denote a core of a $G^r$ in $\mathcal{G}_2$. Clearly $G$ intersects $A\cup B$ in an independent set, let us denote it by $U$. Then $G$ intersects $C$ in a graph of chromatic number at least $k-1$. Therefore, $G$ has to contain some edges such that the corresponding hyperedges of $G^r$ intersect $C$ in non-edges of $\mathcal{G}_0$. Let $Z$ be the set of these edges. Then the corresponding hyperedges of $G^r$ each contain an element from $A$, thus there is an intersection outside the core, a contradiction.
\end{proof}

Observe that the size of $\mathcal{G}_2$ increases with $s$, and thus it is greater than that of $\mathcal{G}_1$ for $s$ sufficiently large. In this case we have a similar dichotomy as above: if we delete a whole color class, then the situation is much simpler. First we describe our construction if this is not the case.
Again, we consider a set of vertices $U$ and two sets of edges $X$ and $Y$ such that $U$ is an independent set not incident to any edge from $X\cup Y$, $X$ and $Y$ are disjoint, deleting $U$ and $X$ from $G$ results in a graph with chromatic number $k-1$, and deleting $U,X,Y$ results in a graph with chromatic number $k-2$. Let $f(U,X,Y)=|X|\binom{n}{r-1}+(|Y|-1)\binom{k-1}{r-1}\left(\frac{n}{k-1}\right)^{r-1}+(s-|X|-|Y|+1)\binom{k-2}{r-1}\left(\frac{n}{k-2}\right)^{r-1}$. Let $h'(G)$ denote the smallest possible value of $f(U,X,Y)$ for sets $U,X,Y$ satisfying the above properties, and $u,x,y$ denote the sizes of those sets $U,X,Y$, respectively. Assume that $s\ge x+y-1$.

Let $\mathcal{G}_2'$ be the following $r$-graph. We take a set $A$ of $x$ vertices, a set $B$ of $y-1$ vertices, a set $C$ of $s-x-y+1$ vertices and a set $D$ of $n-s$ vertices. We take on $D$ a complete $(k-2)$-partite $(r-1)$-graph $\mathcal{G}_0$ and a complete $(k-1)$-partite $(r-1)$-graph $\mathcal{G}_0'$. We take as hyperedges of $\mathcal{G}_2'$ the $r$-sets that contain a vertex from $A$ and $r-1$ vertices from $D$, the $r$-sets that contain a vertex from $B$ and a hyperedge from $\mathcal{G}_0'$ and the $r$-sets that contain a vertex from $C$ and a hyperedge from $\mathcal{G}_0$.

\begin{prop}
$\mathcal{G}_2'$ is $G^r$-free. 
\end{prop}
\begin{proof}[\bf Proof]
Assume otherwise and let $G$ denote a core of a $G^r$ in $\mathcal{G}_2'$. Clearly $G$ intersects $A\cup B\cup C$ in an independent set, let us denote it by $U$. Then $G$ intersects $D$ in a graph of chromatic number at least $k-1$. Therefore, $G$ has to contain some edges such that the corresponding hyperedges of $G^r$ intersect $D$ in non-edges of $\mathcal{G}_0$. Among these edges, let $Y$ be the set of those edges such that the corresponding hyperedges of $G^r$ that intersect $D$ belong to $\mathcal{G}_0'$ and let $X$ be the set of edges such that the corresponding hyperedges of $G^r$ that intersect $D$ belong to neither $\mathcal{G}_0$ nor $\mathcal{G}_0'$. Then the hyperedges of $G^r$ that correspond to $X$ each contain a vertex from $A$. The hyperedges of $G^r$ that correspond to $Y$ each contain a vertex from $B$. Therefore, $|X|\le x$ and $|Y|\le y-1$. Observe that $U,X,Y$ satisfy the partition properties defined above, but $f(U,X,Y)<h'(G)$, a contradiction.
\end{proof}

The simplest case seems to be when deleting $U$ results in a $(k-1)$-colorable graph, and then $Z$ consists of the edges between two color classes. In this case, we do not have to worry about $\mathcal{G}_2'$, but $\mathcal{G}_1$ still might be larger than $\mathcal{G}_2$. To avoid this complication, we may assume that $s$ is sufficiently large. In this case we believe that $\mathcal{G}_2$ gives the asymptotically correct value of $\ex(n,\{G^r,M_{s+1}^r\})$. 

\begin{conjecture}\label{bla} Let $\chi(G)=k>r$ and assume that there is an independent set $U$ of $G$ such that deleting $U$ from $G$ results in a graph $G_0$ of chromatic number $k-1$, and there are two color classes of $G_0$ such that there are $m'(G)$ edges between them. If $s$ is sufficiently large, then
    \[\ex(n,\{G^r,M_{s+1}^r\})=(m'(G)-1)\binom{n}{r-1}+(s-m'(G)+1)\binom{k-2}{r-1}\left(\frac{n}{k-2}\right)^{r-1}+o(n^{r-1}).\]
\end{conjecture}

\bigskip
\textbf{Funding}: The research of Gerbner is supported by the National Research, Development and Innovation Office - NKFIH under the grants FK 132060 and KKP-133819. The research of Tompkins is supported by the National Research, Development and Innovation Office - NKFIH under the grant K135800. The research of Zhou is supported by the China Scholarship Council (No. 202406890088).


\begin{thebibliography}{99}


\bibitem{af}  N. Alon, P. Frankl, Tur\'{a}n graphs with bounded matching number. \textit{Journal of Combinatorial Theory, Series B}, \textbf{165}, 223--229, 2024.




\bibitem{Er2} P. Erd\H{o}s, On extremal problems of graphs and generalized graphs. \textit{Israel J. Math}, \textbf{2}, 183--190, 1964. 

\bibitem{Er1} P. Erd\H{o}s, A problem on independent $r$-tuples. \textit{Ann. Univ. Sci. Budapest}, \textbf{8}, 93--95, 1965.



\bibitem{Fr1} P. Frankl, Improved bounds for Erd\H{o}s' matching conjecture. \textit{Journal of Combinatorial Theory, Series A}, \textbf{120}(5), 1068--1072, 2013.

\bibitem{Fu1} Z. F\"{u}redi, Tur\'{a}n Type problems. \textit{Surveys in combinatorics}, London Math. Soc. Lecture Notes Ser., \textbf{166}, Cambridge Univ. Press, Cambridge, 253--300, 1991.

\bibitem{Fu2} Z. F\"{u}redi, M. Simonovits, The history of degenerate (bipartite) extremal problems. \textit{Erd\H{o}s Centennial}, Bolyai Society Math. Studies, vol. \textbf{25}, Springer, 169--264, 2013.

\bibitem{gerbner} D. Gerbner, On Tur\'{a}n problems with bounded matching number. \textit{Journal of Graph. Theory}, 1--7, 2023.


\bibitem{Ke1} P. Keevash, Hypergraph Tur\'{a}n problems. \textit{Surveys in combinatorics}, 83--139, 2011, London Math. Soc. Lecture Note Ser., \textbf{392}, Cambridge University Press, Cambridge, 83--139, 2011.


\bibitem{mub} D. Mubayi. A hypergraph extension of Tur\'an's theorem.
{\it Journal of Combinatorial Theory, Series B}, \textbf{96}, 122--134, 2006.

\bibitem{mubver} D. Mubayi, J. Verstra\"ete. A survey of Tur\'an problems for expansions. \textit{Recent Trends in Combinatorics}, 117--143, 2016.

\bibitem{zhy} J. Zhou, X. Yuan, Linear Tur\'{a}n problems with bounded matching number in hypergraphs. \textit{manuscript}, 2024. 





\end{thebibliography}
\end{document}